\documentclass{amsart}

\usepackage{amsmath}
\usepackage{amssymb}
\usepackage{enumerate}
\usepackage{qtree}
\usepackage{youngtab}
\usepackage{tikz}
\usetikzlibrary{shapes.geometric}

\title[Representations on the Boundary of Rooted Trees]
 {Groups Acting on Rooted Trees and Their Representations on the Boundary} 

\author[S. Kionke]{Steffen Kionke}

\address{Karlsruher Institut f\"ur Technologie \\ Fakult\"at f\"ur Mathematik \\
  Institut f\"ur Algebra und Geometrie \\ Englerstr.2 \\
  76131 Karlsruhe \\ Germany.}

\email{steffen.kionke@kit.edu}

\thanks{The research was supported by DFG grant KL 2162/1-1.}

\date{\today}

\theoremstyle{plain}
\newtheorem{theorem}{Theorem}
\newtheorem*{theorem*}{Theorem}
\newtheorem{lemma}{Lemma}
\newtheorem{corollary}{Corollary}
\newtheorem{proposition}{Proposition}

\theoremstyle{definition}
\newtheorem{definition}{Definition}

\newtheorem{remark}{Remark}
\newtheorem{example}{Example}

\DeclareMathOperator{\id}{Id}
\DeclareMathOperator{\Hom}{Hom}
\DeclareMathOperator{\End}{End}
\DeclareMathOperator{\Aut}{Aut}
\DeclareMathOperator{\St}{St}
\DeclareMathOperator{\GL}{GL}
\DeclareMathOperator{\SL}{SL}
\DeclareMathOperator{\chr}{char}
\DeclareMathOperator{\Ind}{Ind}
\DeclareMathOperator{\Irr}{Irr}
\DeclareMathOperator{\Sym}{Sym}

\providecommand{\calT}{\mathcal{T}}
\providecommand{\calG}{\mathcal{G}}
\providecommand{\calB}{\mathcal{B}}
\providecommand{\Cinf}{C^\infty}
\providecommand{\smInd}{\underline{\Ind}}

\providecommand{\frakp}{\mathfrak{p}}

\providecommand{\bbN}{\mathbb{N}}
\providecommand{\bbR}{\mathbb{R}}
\providecommand{\bbP}{\mathbb{P}}
\providecommand{\bbQ}{\mathbb{Q}}
\providecommand{\bbZ}{\mathbb{Z}}
\providecommand{\bbF}{\mathbb{F}}
\providecommand{\bbC}{\mathbb{C}}

\providecommand{\triv}{1\!\!1}
\providecommand{\fat}[1]{\boldsymbol{#1}}

\newcommand{\up}[1]{\:^#1 }

\renewcommand{\epsilon}{\varepsilon}
\renewcommand{\phi}{\varphi}

\begin{document}

\begin{abstract}
  We consider groups that act on spherically symmetric rooted trees and study the associated
  representation of the group on the space of locally constant functions on the boundary of
  the tree. We introduce and discuss the new notion of locally $2$-transitive actions.
  Assuming local $2$-transitivity our main theorem yields a
  precise decomposition of the boundary representation into irreducible constituents.
 
  The method can be used to study Gelfand pairs and enables us to answer a question of
  Grigorchuk. To provide examples, we analyse in detail the local
  $2$-transitivity of GGS-groups. Moreover, our results can be used to determine explicit formulae for
  zeta functions of induced representations defined by Klopsch and the author. 
\end{abstract}

\maketitle

\section{Introduction}

The objects of interest in this article are groups which act on spherically symmetric rooted
trees.  This means, they act on a tree $\calT$ with a distinguished vertex $\varepsilon$,
called the root, such that every two vertices having same distance from the root have the same
finite number of neighbours.  The set $L_n$ of vertices which have distance $n$ from the root
is called the $n$-th level of the tree. Whenever a group $G$ acts on $\calT$ by
root-preserving automorphisms, then $G$ also acts on the boundary $\partial \calT$ of $\calT$.
The boundary is the set of all infinite rays starting at the root and it is equipped in a
natural way with a profinite topology.
The action of $G$ on $\partial \calT$ gives rise to representations of $G$ on spaces of
functions on the boundary.
These representations should contain information about
the action of $G$ on $\calT$ and, conversely, it would be desirable to understand such
representations via the geometric properties of the action of $G$ on
$\calT$.

An interesting and natural example of such a representation is the unitary representation $\rho^{(2)}$
of $G$ on $L^2(\partial \calT, \mu)$, where
$\mu$ is the $\Aut(\calT)$-invariant probability measure $\mu$ on the boundary.
This representation decomposes as a Hilbert space direct sum of infinitely many finite dimensional irreducible representations; c.f.~\cite{BarGri2000}.
By contrast, for certain weakly branch groups the associated Koopman
representation of $G$ on $L^2(\partial\calT,\nu)$ is irreducible for most Bernoulli measures
$\nu \neq \mu$; see \cite{DudkoGrigorchuk2017}.
In a similar vein, the representation of a weakly branch group $G$ on $\ell^2(G/\St_G(\xi))$
for a point $\xi \in \partial\calT$ is irreducible; see \cite{BarGri2000}.

In  $L^2(\partial \calT, \mu)$, however, the locally constant complex-valued functions $\Cinf(\partial \calT, \bbC)$ form a
dense subrepresentation and the structure of the unitary
representation $\rho^{(2)}$ is completely determined by the structure of the representation
$\rho_{\partial,\bbC}$ on $\Cinf(\partial \calT, \bbC)$.
We will review this close relationship in Section~\ref{sec:unitary-vs-smooth} below and we will use it to translate
results concerning $\rho^{(2)}$ to the representation $\rho_{\partial,\bbC}$.
Our objective is to describe a decomposition of $\rho_{\partial,\bbC}$
into irreducible constituents in terms of geometric data.

One way to approach this problem is via  \emph{distance transitive} (or \emph{$2$-point
  homogeneous}) actions on metric spaces; these are isometric actions which are transitive on
equidistant pairs of points.  It is well-known that a distance transitive action of a finite
group $H$ on a finite metric space $(X,d)$ yields a \emph{multiplicitiy-free} representation of $H$ on
$\bbC[X]$;
see \cite[Ex.~2.5]{CSST2007}.  In other words, $(H, \St_H(x))$ is a \emph{Gelfand pair} for
every $x \in X$.  We refer to \cite{Bump, DieudonneAna6, Gr91} for an introduction
to Gelfand pairs.  Suffice it to mention that finding and understanding
Gelfand pairs is an important issue in the representation theory of groups with applications
to statistics and probability theory; see for instance \cite{CSST2007, Diaconis}. Distance
transitive actions of groups on the boundary of finite trees were frequently used to construct
Gelfand pairs of finite groups; see \cite{DD2007, DD2012, Letac1982}.
Five examples of Gelfand pairs arising from groups acting on infinite rootes trees (e.g.\ the Grigorchuk group)
are discussed in \cite{BarGri2001}.

In the appendix of \cite{BeHa03} Bekka, de la Harpe and Grigorchuk refined the
distance-transitivity method and obtained an explicit description of the irreducible
constituents of the boundary representation $\rho_{\partial,\bbC}$ under the assumption that
$G$ acts distance transitively on every level.  In this case the boundary representation
$\rho_{\partial,\bbC}$ contains, apart from the trivial representation, exactly one
irreducible summand of dimension $|L_n|-|L_{n-1}|$ for all $n \in \bbN$. The account in
\cite{BeHa03} is based on unitary representations and can be translated via the remarks in
Section~\ref{sec:unitary-vs-smooth} to apply to~$\rho_{\partial,\bbC}$.  A direct
argument is contained in \cite[Thm.~6.3]{KiKl}.  As a consequence, the representation is
multiplicity-free and gives rise to a Gelfand pair $(\overline{G},P_\xi)$ of profinite groups
where $\overline{G}$ is the closure of $G$ in the automorphism group $\calG = \Aut(\calT)$ and
$P_\xi$ is a parabolic subgroup, i.e.\ the stabiliser in $\overline{G}$ of a point
$\xi\in \partial\calT$.

Here we generalise the result of Bekka, de la Harpe and Grigorchuk by assuming only a weak
version of distance transitivity and by working over a general field $F$ of suitable
characteristic.  In fact, the representation on locally constant functions can be defined over
an arbitrary base field~$F$. Indeed, the space $\Cinf(\partial\calT,F)$ of locally constant
$F$-valued functions on $\partial\calT$ carries a representation of $G$ by imposing
$$\rho_{\partial,F}(g)(f)(\xi) = f(g^{-1}\xi)$$
for all $g\in G$, $f \in \Cinf(\partial\calT,F)$ and $\xi\in \partial \calT$.  The nature of
the representation depends strongly on the characteristic of the underlying field and on
whether it divides the \emph{Steinitz order} of $\overline{G}$. Here we say that the
characteristic of $F$ divides the Steinitz order of $\overline{G}$, if it divides the order of
some finite continuous quotient of $\overline{G}$.  One extremal case, where $\calT$ is the
$p$-regular tree, $F$ is of characteristic $p$ and $G$ is contained in a Sylow pro-$p$
subgroup of $\Aut(\calT)$ is studied in \cite{GLNS2016}.  Here, however, we will consider the
opposite case and assume that the characteristic of $F$ \emph{does not} divide the Steinitz
order of $\overline{G}$.  In this case, the representation $\rho_{\partial,F}$ is semi-simple,
i.e., decomposes as a direct sum of finite dimensional irreducible representations.

  A spherically transitive action of $G$ on $\calT$ will be called
\emph{locally $2$-transitive}, if for all distinct vertices $u,v \in L_n$ the common
stabiliser $\St_G(u)\cap \St_G(v)$ acts transitively on $D(u)\times D(v)$. Here $D(u)$ denotes
the set of descendants of $u$, that is, the neighbours of $u$ which lie in the next level
$L_{n+1}$. In fact, every distance transitive action is locally $2$-transitive; see Lemma
\ref{lem:distance-trans}.

For a vertex $v$ let $\Irr_F^0(v)$ denote the set of non-trivial irreducible $F$-representations
$\pi$ of $\St_G(v)$ which occur with non-zero multiplicity $m(\pi,\theta_{v,F})$ in the permutation
representation $\theta_{v,F}$ of $\St_G(v)$ on $F[D(v)]$.  Our main result provides a precise
description of the irreducible constitutens of $\rho_{\partial,F}$.

\begin{theorem}\label{thm:decomposition}
  Let $G \leq \calG$ act locally $2$-transitively on $\calT$.
  Assume that $F$ is a field whose characteristic does not divide the Steinitz order of $\overline{G}$.

  The induced representation $\Ind_{\St_G(v)}^G(\pi)$ is irreducible for every $\pi \in \Irr_F^0(v)$.
  The decomposition of
  $\rho_{\partial,F}$ into distinct irreducible constituents is given by
  \begin{equation*}
    \rho_{\partial,F} \cong \triv_F \oplus \bigoplus_{n\in \bbN_0} \bigoplus_{\pi\in\Irr_F^0(v_n)}m(\pi,\theta_{v_n,F})\;
    \Ind_{\St_G(v_n)}^G(\pi),
  \end{equation*}
  where $v_n$ denotes an arbitrary vertex of level $n$ and $\triv_F$ is the trivial representation on $F$.
  \end{theorem}

  In particular, the representation $\rho_{\partial,F}$ is multiplicity-free if and only if
  $\theta_{v,F}$ is multiplicity-free for every vertex $v$.  For $F= \bbC$ this implies the
  following local-global principle.

\begin{theorem}[Local-global principle]\label{thm:local-global} Let
  $G\leq_c \calG$ be a closed subgroup which acts locally $2$-transitively on $\calT$.  Let
  $P_\xi \leq G$ be a parabolic subgroup.  Then $(G,P_\xi)$ is a Gelfand pair if and only if
  $(\St_G(v),\St_G(u))$ is a Gelfand pair for every vertex $v$ and descendant $u\in D(v)$.
\end{theorem}
On one hand, the local-global principle can be used to construct Gelfand pairs of profinite
groups.  Indeed, if the number of descendants of every vertex is a prime or the square of a
prime, then the local pairs $(\St_G(v),\St_G(u))$ are Gelfand pairs; for details we refer to
Corollary~\ref{cor:prime-square}.  On the other hand, we use the local-global principle to construct
branch groups $G \leq_c \calG$ such that $(G,P_\xi)$ is not a Gelfand pair for every parabolic
subgroup $P_\xi$ of $G$; see Example \ref{ex:wreath-not-Gelfand}.  This answers a question of
Grigorchuk; see Problem 10.1 in~\cite{Gri2005}.

In Section \ref{sec:GGS} we study the local $2$-transitivity of \emph{GGS-groups}, named after
the famous finitely generated, infinite torsion groups of Grigorchuk \cite{Gri1980} and Gupta
and Sidki \cite{GuSi1983}.  These groups attracted considerable interest recently and the
reader is refered to the monograph \cite[Sec.~2.3]{BarGriSun2003} and the articles
\cite{FAZR2014,FAGUA2017,Pervova2007} for more information.  Here we consider (generalised)
GGS-groups which act on a $p^k$-regular rooted tree. In this case every vertex in the tree
$\calT$ has exactly $p^k$ descendants for a prime number $p$. Theorem \ref{thm:GGS-2-trans}
characterises the defining vectors with locally $2$-transitive associated GGS-groups.  It
follows from Remark \ref{rem:p-always-aperiodic} that our criterion is always satisfied if $p$
is an odd prime and $k=1$. The following result is a consequence of
Theorem~\ref{thm:GGS-2-trans}.
\begin{theorem}\label{thm:GGSsimpleversion}
  Let $p$ be an odd prime number and let $\calT$ be the $p$-regular rooted tree.
  Every GGS-group $G_{\fat{e}}\leq \Aut(\calT)$ acts locally $2$-transitively.
  Moreover, $(\overline{G}_{\fat{e}},P_\xi)$ is a Gelfand pair for
  every parabolic subgroup $P_\xi \leq_c \overline{G}_{\fat{e}}$.
\end{theorem}
 This result confirms a conjecture
mentioned by Bartholdi and Grigorchuk in the very last sentence of \cite{BarGri2001}.
We remark that the unique GGS-group which acts on the rooted binary tree does not act locally
$2$-transitively (see Remark \ref{rem:Dinfty}).  However it is still true that
$(\overline{G}_{\fat{e}},P_\xi)$ is a Gelfand pair.

\subsection{Applications to zeta functions of representations}

Let $G$ be a profinite group and consider, for every $n$, the number $r_n(G)$ of irreducible
$n$-dimensional complex continuous representations of $G$.  We assume that $G$ is
\emph{representation rigid}, this means, that the numbers $r_n(G)$ are all finite.  The
paradigm of representation growth postulates that the asymptotic behaviour of the sequence
$(r_n(G))_{n\in \bbN}$ contains information on the structure of the group $G$. In recent years
this line of investigation stimulated a considerable amount of new results; see e.g.
\cite{AvKlOnVo13, AvKlOnVo16a, Ja06, LaLu08, LuMa04} and the survey~\cite{Kl2013} as well as references therein.
Here once again groups which act on regular rooted trees provide a class of examples with remarkable properties
-- for self-similar branched groups this was discussed by  Bartholdi in \cite{Bartholdi2017}.
His results apply, in particular, to the Gupta-Sidki $p$-groups which will also be considered in Section~\ref{sec:GGS} below. The special case of iterated wreath products was studied already in~\cite{Bartholdi-delaHarpe2010}.

In many examples the numbers $r_n(G)$ grow at most polynomially in $n$ and it is natural to
investigate the asymptotic properties using the Dirichlet series
\begin{equation*}
    \zeta_G(s) = \sum_{n=1}^\infty \frac{r_n(G)}{n^s},
\end{equation*}
which is called the \emph{representation zeta function}.
For certain self-similar branch groups Bartholdi's method \cite{Bartholdi2017} allows
the efficient computation of the coefficients $r_n(G)$ using a functional equation of $\zeta_G$.
However, it is worth noting, that in general explicit
formulae for $\zeta_G$ are quite difficult to obtain; see \cite{AvKlOnVo13,
  AvKlOnVo16a} for an extensive treatment of compact $p$-adic Lie groups of type
$A_2$ (such as $\SL_3(\bbZ_p)$).

\medskip

Facing this difficulty it might be fruitful to break the problem into smaller pieces by
counting only certain subsets of all representations.
In the recent preprint \cite{KiKl} B.~Klopsch and the author developed such a
more flexible approach to representation zeta functions which is based on the following observation.
The representation zeta function at $s-1$ is

\begin{equation*}
    \zeta_G(s-1) = \sum_{n=1}^\infty \frac{r_n(G)n}{n^s} = \sum_{\pi \in \Irr_\bbC(G)} \frac{\dim(V_\pi)}{\dim(V_\pi)^s}
  \end{equation*}
  and the dimension $\dim_\bbC (V_\pi)$ of an irreducible complex representation $(\pi,V_\pi)$
  of $G$ is equal to the multiplicity of $\pi$ in the regular representation of $G$ on the
  space of locally constant functions $\Cinf(G,\bbC)$. Equivalently, it is the multiplicity in
  the regular representation of $G$ on $L^2(G)$ with respect to the Haar measure; see
  Section~\ref{sec:unitary-vs-smooth}.  Let $(\rho,V_\rho)$ be an admissible smooth
  representation of $G$; see Section~\ref{sec:rep-profinite} for the definition of smooth
  representations. Such a representation decomposes into a direct sum
\begin{equation*}
   V_\rho = \bigoplus_{\pi\in\Irr_{\bbC}(G)} m(\pi,\rho) V_\pi 
 \end{equation*}
 over the finite dimensional irreducible complex representations of $G$
 which occur with finite multiplicities $m(\pi,\rho) \in \bbN_0$.
 Under suitable assumptions it is natural to investigate
 the Dirichlet series
 \begin{equation*}
    \zeta_\rho(s) =  \sum_{\pi \in \Irr(G)} \frac{m(\pi,\rho)}{\dim_\bbC( V_\pi)^s},
  \end{equation*}
  which will be called the zeta function of the representation $(\rho, V_\rho)$;
  for the details we refer to \cite{KiKl}.
  As observed above, the zeta function of the regular representation equals $\zeta_G(s-1)$.
  Here the results of \cite{KiKl} will not be used. It is sufficient to know that
  this setting provides a wealth of examples, where explicit formulae for zeta functions of representations
  can actually be determined; see \cite[Sect.~7]{KiKl}.

  \medskip

  Theorem \ref{thm:decomposition} provides a new way to
  obtain explicit formulae for zeta functions of representations.
  We return to the situation where
  $G \leq_c \calG$ is a closed locally $2$-transitive subgroup of the automorphism group
  of a spherically symmetric rooted tree $\calT$.
  Theorem \ref{thm:decomposition}
  yields a formula for the zeta function of the representation $\rho_{\partial,\bbC}$ on the boundary
  \begin{equation}\label{eq:zeta-boundary-formula}
    \zeta_{\rho_{\partial,\bbC}}(s) =
    1 + \sum_{n\in \bbN_0} \sum_{\pi \in \Irr^0_\bbC(v_n)} \frac{m(\pi,\theta_{v_n,\bbC})}{|L_n|^s \dim(V_\pi)^s}.
  \end{equation}
  Since $\partial \calT \cong G/P$ for a parabolic subgroup $P$, the representation
  $\rho_{\partial,\bbC}$ is isomorphic to the smoothly induced representation
  $\smInd_P^G(\triv_\bbC)$; see Section~\ref{sec:rep-profinite} for the definition of smooth
  induction.

  In certain cases, this makes it possible to use formula \eqref{eq:zeta-boundary-formula} to
  determine the zeta function of a smoothly induced representation.  Let $G$ be any profinite
  group with a closed subgroup $P \leq_c G$.  If one can realise the homogeneous space $G/P$
  as the boundary of a spherically symmetric rooted tree $\calT$ with a locally $2$-transitive
  action of $G$, then formula \eqref{eq:zeta-boundary-formula} describes the zeta function of
  $\smInd_P^G(\triv_\bbC)$.  We give an example to illustrate the usefulness of this approach in
  Section \ref{sec:repr-zeta-functions}.

  \section{Preliminaries from representation theory}
  Here we review some preliminary results from representation theory.
  Let $F$ be a field.
  Recall that for an abstract group $H$, a subgroup $S\leq H$ and an $F$-representation $(\pi,V_\pi)$ of $S$,
  the representation \emph{induced} from $\pi$ is the representation of $H$ on the space
 \begin{equation*}
  \Ind_S^H(\pi) = \{ f \colon H \to V_\pi \mid f(s x) =  \pi(s) f(x) \text{ for all } s \in S \}.
  \end{equation*}
  by the right regular action, i.e.\ $\rho(h)(f)(x) := f(xh)$ for all $f \in \Ind_S^H(\pi)$
  and all $x,h \in H$. We remark that sometimes
  this is also called the coinduced representation.

 \subsection{Representation theory of finite groups}
  In this short subsection $H$ denotes a finite group and
 $F$ denotes a field whose characteristic \emph{does not} divide the order of $H$. Under these
assumptions Maschke's theorem implies that the group algebra $F[H]$ is a semisimple $F$-algebra (see
e.g.\ \cite[XVIII Thm.~1.2]{LangAlgebra}). In particular, every representation of $H$ on an $F$-vector space
decomposes as a direct sum of finite dimensional irreducible representations.

Moreover, we recall that every action of $H$ on a finite set $A$ gives rise to a representation
$\rho_{A,F}$ of $H$ on the vector space
\begin{equation*}
    F[A] = \Bigl\{\sum_{a \in A} c_a a \mid c_a \in F \Bigr\}
  \end{equation*}
  of formal $F$-linear combinations of the elements in $A$ by imposing $\rho_{A,F}(g)(a) = ga$.
 We collect two preliminary results from
the representation theory of finite groups.
\begin{lemma}\label{lem:diagonal-transitive}
  Assume that $H$ acts on finite sets $A$ and $B$.  If the
  diagonal action of $H$ on $A\times B$ is transitive, then
 \begin{equation*}
   \dim_F \Hom_{F[H]}(F[A],F[B]) = 1.
 \end{equation*}
 In particular, the trivial representation is the only irreducible
 $F$-representation of $H$ which occurs in both $F[A]$ and $F[B]$.
\end{lemma}
\begin{proof}
  The space of $F$-linear maps $\Hom_F(F[A],F[B])$ is equipped with an $H$-action via
  $(h\cdot\varphi)(a) := h\varphi(h^{-1}a)$ for all $a \in A$.
  We note that the $H$-invariants in this space are exactly the $H$-equivariant linear maps.
  Since $A$ is finite, there is an isomorphism of $F[H]$-modules
  \begin{equation*}
    \Hom_F(F[A],F[B]) \cong F[A\times B]
  \end{equation*}
  We conclude that
  \begin{equation*}
    \Hom_{F[H]}(F[A],F[B]) \cong   \Hom_F(F[A],F[B])^H \cong F[A\times B]^H
  \end{equation*}
  where $F[A\times B]^H$ denotes the space of $H$-invariants.
  Since $H$ acts transitively on $A\times B$, the space of $H$-invariants in $F[A\times B]$ is
  one-dimensional.
\end{proof}

\begin{lemma} \label{lem:irreducibilityCriterion} Let $S \leq H$ be a subgroup.
  Given an irreducible $F$-representation $(\pi,V_\pi)$ of $S$ so that
  $$\dim_F \End_{F[S]}(\pi) = \dim_F \End_{F[H]}(\Ind_S^{H}(\pi)),$$
  then the induced representation $\Ind_S^{H}(\pi)$ is irreducible.
\end{lemma}
\begin{proof}
  Since $\pi$ is irreducible, the endomorphism algebra $E = \End_{F[S]}(\pi)$ is a division
  $F$-algebra.  The induction functor provides a homomorphism of
  $F$-algebras $J \colon E \to \End_{F[H]}(\Ind_S^{H}(\pi))$.  Since $E$ is
  a division algebra the homomorphism $J$ is injective. The assumption on the dimensions
  implies that $J$ is actually an isomorphism.  In particular,
  $ \End_{F[H]}(\Ind_S^{H}(\pi))$ is a division algebra and, hence, does not
  contain non-trivial idempotent elements.
  By Maschke's theorem the representation $ \End_{F[H]}(\Ind_S^{H}(\pi))$
  is semi-simple, i.e.\ every subrepresentation is a direct summand.
  Therefore the non-existence of idempotent elements
  implies that $\Ind_S^{H}(\pi)$ is irreducible.
\end{proof}

\subsection{Representation theory of profinite groups}\label{sec:rep-profinite}
Let $H$ denote a profinite group and let $F$ be a field.  A representation $(\rho, W_\rho)$ of
$H$ on some (possibly infinite dimensional) $F$-vector space $W_\rho$ is called \emph{smooth}
if every $w \in W_\rho$ has an open stabilizer in $H$.  This means, that $W_\rho$ is the union
of finite dimensional subrepresentations on which the action of $H$ factors through some
finite continuous quotient of $H$.  In particular, a finite dimensional representation is
smooth exactly if it factors over some finite continuous quotient of $H$.  We will impose a
condition on the characteristic of $F$, which ensures that smooth representations decompose
into direct sums of finite dimensional irreducible representations.
 \begin{definition}
   We say that the characteristic of $F$ \emph{does not divide the Steinitz order of $H$},
   if $\chr(F)$ does not divide the order of any finite continuous quotient of $H$.
 \end{definition}
 This definition alludes to the Steinitz order of a profinite group in terms of Steinitz numbers.
 We refer to \cite[Chapter~2]{WilsonProfinite} for the definition of
 the Steinitz order of a profinite group.
 Here this connection will not be important.

 Assume now that $\chr(F)$ does not divide the Steinitz order of $H$.  The key
 observation is that, since smooth representations of $H$ are unions of representations of
 finite quotients, Maschke's theorem implies that every smooth representation $(\rho,W_\rho)$ of $H$ over $F$
 is a direct sum of isotypic components, i.e.
 \begin{equation*}
   W_\rho = \bigoplus_{\pi \in \Irr_F(H)} W_\rho(\pi)
 \end{equation*}
 where the sum runs over all smooth irreducible $F$-representations of $H$.  The isotypic
 component $W_\rho(\pi)$ isomorphic to a direct sum of copies of $\pi$: the number
 $m(\pi,\rho)$ of copies is called the \emph{multiplicity} of $\pi$ in $\rho$.
 The representation $\rho$ is said to be \emph{admissible},
 if all multiplicities are finite;
 in other words, all isotypic components are finite dimensional.

 \medskip
 
 For convenience we recall the concept of smooth induction.
 Let $S \leq_c H$ be a closed subgroup and
let $(\pi, V_\pi)$ be a smooth representation of $S$.
The representation smoothly induced from $\pi$ is a representation of $H$
on the space
\begin{equation*}
  \smInd_S^H(\pi) = \{ f \colon H \to V_\pi \mid f \text{ loc.\ constant and } f(s x) =  \pi(s) f(x) \text{ for all } s \in S \}.
\end{equation*}
The representation of $H$ is given by the right regular action, i.e.\ $\rho(h)(f)(x) := f(xh)$
for all $f \in \smInd_S^H(\pi)$ and all $x,h \in H$.  In particular, if $\pi$ is the trivial
representation on a one dimensional $F$-vector space, then $\smInd_S^H(\pi)$ is the space
$\Cinf(H/S,F)$ of locally constant $F$-valued functions on the homogeneous space $H/S$.
The representation of $H$ on $\Cinf(H/S,F)$ is always admissible.

\subsection{Unitary versus smooth representations}\label{sec:unitary-vs-smooth}
Let $\overline{G}$ be a profinite group and let $G \leq \overline{G}$ be a dense subgroup.  We
consider the homogeneous space $X = \overline{G}/P$ for a closed subgroup
$P \leq_c \overline{G}$.  The pushforward measure $\mu$ of the normalized Haar measure on
$\overline{G}$ is a regular $\overline{G}$-invariant Borel probability measure on $X$.
We consider the representation of $\overline{G}$ on the complex Hilbert space $L^2(X,\mu)$ defined
by $(\rho(g)h)(x)= h(g^{-1}x)$ for all $h \in L^2(X,\mu)$ and $g \in \overline{G}$. 
In this
situation, the properties of the representation $\rho|_G$ of $G$ on
$L^2(X,\mu)$ are essentially determined by the representation of $\overline{G}$ (or $G$) on the space
$\Cinf(X,\bbC)$ of locally constant $\bbC$-valued functions on $X$.
This relation seems to be
well-known, however, we could not find a good reference. We include a short discussion,
since this method can be used to interpret our results in the framework of unitary representations.

\begin{proposition} In the situation above the following statements hold.
  \begin{enumerate}
  \item\label{it:dense} The subspace $\Cinf(X,\bbC)$ is dense in $L^2(X,\mu)$.
  \item\label{it:same-irreducibles} Every closed irreducible $G$-subrepresentation of $L^2(X,\mu)$
    is $\overline{G}$-stable and contained in $\Cinf(X,\bbC)$.
  \item\label{it:isotypic-decomposition} Let $\Cinf(X,\bbC) = \bigoplus_{\vartheta \in \Irr(\overline{G})} V(\vartheta)$
    be the decomposition of the $\overline{G}$-representation on $\Cinf(X,\bbC)$ into isotypic components.
    The isotypic components $V(\vartheta)$ are finite dimensional and pairwise orthogonal in $L^2(X,\mu)$.
    In particular, 
    $L^2(X,\mu)$  is isomorphic to the Hilbert
    space direct sum $\widehat{\bigoplus}_{\vartheta \in \Irr(\overline{G})} V(\vartheta)$.
  \end{enumerate}
\end{proposition}
\begin{proof}
  First we consider \eqref{it:dense}. Every $f\in \Cinf(G,\bbC)$ provides a
  convolution operator $T_f \in \calB(L^2(X,\mu))$ given by
  $T_f(h)(x) = \int_{\overline{G}} f(g)h(gx) dg$ with respect to the normalized Haar measure on
  $\overline{G}$.  Since $f$ is locally constant, the function $T_f(h)$ is locally constant
  for every $h \in L^2(X,\mu)$.  Let $\varepsilon > 0$ and $h\in L^2(X,\mu)$ be given.
  The multiplication map $\overline{G} \times L^2(X,\mu) \to L^2(X,\mu)$ is
  continuous, hence there is
  an open normal subgroup $K \trianglelefteq_o \overline{G}$ such that $\Vert h-\rho(k)h \Vert_2 < \varepsilon$
  for all $k \in K$.
  Let $f\in \Cinf(G,\bbR)$ be a positive function supported in $K$ with total mass one, then
  \begin{equation*}
  \Vert T_f(h) - h \Vert_2  \leq \int_{\overline{G}} f(k) \Vert h- \rho(k^{-1})h \Vert_2 dk < \varepsilon
\end{equation*}
and we conclude that $\Cinf(X,\bbC)$ is dense in $L^2(X,\mu)$.

Let $V \subseteq L^2(X,\mu)$ be a closed irreducible $G$-subrepresentation. It follows from
the continuity of the map $\overline{G} \times L^2(X,\mu) \to L^2(X,\mu)$, that $V$
is also $\overline{G}$-stable. By the Peter-Weyl theorem every irreducible
unitary representation of the compact group $\overline{G}$ is finite dimensional; see \cite[Corollary 5.4.2]{Kowalski}. Moreover,
$\overline{G}$ is profinite and the no-small-subgroup argument shows that the representation of $\overline{G}$ on $V$
factors over some finite continuous quotient. We deduce that $V$ lies in $\Cinf(X,\bbC)$.

Finally, let  $\Cinf(X,\bbC) = \bigoplus_{\vartheta \in \Irr(\overline{G})} V(\vartheta)$
be the decomposition into isotypic components as in \eqref{it:isotypic-decomposition}.
Since the representation is admissible, the isotypic components are finite dimensional.
For the fact that distinct isotypic components are orthogonal we refer to \cite[Lemma 3.4.21]{Kowalski}.
\end{proof}

 \section{Locally $2$-transitive actions on rooted trees }

\subsection{Notation}
Let $(X_i)_{i\in \bbN}$ be a sequence of finite sets. Throughout, $m_i = |X_i|$ denotes the
cardinality of $X_i$ and we assume $m_i \geq 2$ for all $i\in \bbN$.  For every $n\in \bbN_0$ we
define
\begin{equation*}
   L_n :=  \prod_{i=1}^n X_i,
\end{equation*}
this is the set of finite sequences $(x_1,\dots,x_n)$ with $x_i \in X_i$. Here we use the
convention that $L_0$ consists of exactly one element: the empty sequence $\epsilon$.  The
\emph{spherically symmetric rooted tree} $\calT = (V(\calT),E(\calT))$ (based on
$(X_i)_{i\in\bbN}$) is a graph on the set of vertices $V(\calT) = \bigsqcup_{n=0}^\infty L_n$
and there is an edge between $u\in L_n$ and $v\in L_{n+1}$ exactly if $u$ is a prefix of $v$.
The distinguished element $\epsilon \in L_0$ is called the root of $\calT$.  In addition, if
$m_i = m$ for all $i$, then we say that $\calT$ is the \emph{$m$-regular} rooted tree.  The set $L_n$
is called the $n$-th \emph{level} of $\calT$.  For a vertex $u\in L_n$ the set of
\emph{descendants} $D(u)$ consists of all vertices in $L_{n+1}$ which are incident to $u$.  If
$u = (x_1,\dots,x_n)$ and $y \in X_{n+1}$, then we write $uy$ for the descendant
$(x_1,\dots,x_n,y)$ of $u$.

The group of root-preserving automorphisms $\calG = \Aut(\calT)$ of $\calT$ is a profinite
group. A neighbourhood base of open normal subgroups is given by the pointwise stabilisers
$\calG(n)$ of the $n$-th level, i.e.,
\begin{equation*}
  \calG(n) = \{g\in \Aut(\calT) \mid \forall u \in L_n \quad gu = u\}.
\end{equation*}
A subgroup $G \leq \calG$ is said to act \emph{spherically transitive} on $\calT$, if $G$ acts
transitively on the level $L_n$ for every $n\in \bbN$.  Clearly, every root-preserving automorphism also
preserves the levels setwise.  For $g \in \calG$ and a vertex $u \in L_n$, we define the
\emph{label} of $g$ at $u$ to be the unique permutation $g_{(u)} \in \Aut(X_{n+1})$ such that
$g(ux) = g(u)g_{(u)}(x)$ for all $x \in X_{n+1}$.

The \emph{boundary} $\partial\calT$ of $\calT$ is the profinite set
\begin{equation*}
  \partial\calT := \prod_{i\in\bbN} X_i,
\end{equation*}
which can be considered as the set of infinite, non-backtracking paths in $\calT$ starting at
the root.  The automorphism group $\calG$ acts continuously on $\partial\calT$.  We observe
that a subgroup $G\leq \calG$ acts spherically transitively on $\calT$ if and only if its
closure $\overline{G} \leq_c \calG$ acts transitively on $\partial\calT$.  The stabiliser of
$u\in L_n$ in $G$ will be denoted by $\St_G(u)$.  The stabiliser $P_\xi$ of a point
$\xi \in \partial\calT$ in $\overline{G}$ is called a \emph{parabolic subgroup} of
$\overline{G}$.  All parabolic subgroups are conjugated in $\overline{G}$, whenever the action
of $G$ is spherically transitive.

\subsection{Local $2$-transitivity: Definition and examples}

\begin{definition}
  A subgroup $G \leq \calG$ acts \emph{locally $2$-transitively} on $\calT$, if $G$ acts spherically
  transitive and for all $n\in \bbN$ and every pair $u,v\in L_n$ of distinct vertices, the group
  $\St_G(u) \cap \St_G(v)$ acts transitively on $D(u) \times D(v)$ via the diagonal action.
\end{definition}
\begin{remark}
  It is worth noting, that a group $G$ acts locally $2$-transitively
  if and only if its closure $\overline{G}$ does.
\end{remark}

We give two examples of actions that are locally $2$-transitive. First, we discuss iterated
wreath products.  For every $i \in \bbN$, let $G_i \subseteq \Aut(X_i)$ be a transitive
subgroup. We denote the collection $(G_i \mid i \in \bbN)$ of these subgroups by $C$. The
iterated wreath product defined by $C$ is the closed subgroup $G_C$ of $\calG$ which consists
of all elements which have only labels in $C$, i.e.
\begin{equation*}
     G_C = \{g \in \calG \mid \forall n \in \bbN_0 \forall u \in L_n\;  g_{(u)}\in G_{n+1}\}.
  \end{equation*}
 
  \begin{lemma}\label{lem:wreath-locally2}
    Let $C$ be a collection of transitive subgroups. The action of the iterated wreath product
    $G_C$ on $\calT$ is locally $2$-transitive.
  \end{lemma}
  \begin{proof}
    It follows immediately from an inductive argument and the fact that every
    $G_i \subseteq \Aut(X_i)$ is transitive, that the action of $G_C$ is spherically
    transitive.  Now, let $u,v \in L_n$ be distinct elements of the $n$-th level. The group
    $\St_G(u)\cap \St_G(v)$ acts on $D(u) \times D(v)$ like $G_{n+1}\times G_{n+1}$, thus the
    action is transitive.
  \end{proof}

  Second, we argue that the notion of locally $2$-transitive actions weakens
  the notion of distance transitive actions.  The shortest-path distance on $\calT$ induces a
  metric $d$ on every level $L_n$. A group $G\leq \calG$ is said to act \emph{distance
    transitively} on $\calT$, if for all $n \in \bbN$ and all $u_1,u_2,v_1,v_2 \in L_n$ with
  $d(u_1,u_2) = d(v_1,v_2)$, there is some $g\in G$ so that $gu_i = v_i$ for all
  $i\in\{1,2\}$.
\begin{lemma}\label{lem:distance-trans}
  Let $G\leq \calG$. If $G$ acts distance transitively on $\calT$, then the action of $G$ is locally
  $2$-transitive.
\end{lemma}
\begin{proof}
  Clearly, a distance transitive action is spherically transitive. Let $u,v\in L_n$ be
  distinct vertices.  Every pair of descendants $u' \in D(u)$ and $v' \in D(v)$ satisfies
  $d(u',v') = d(u,v) +2$. Therefore, if $G$ acts distance transitively on $\calT$, then
  $\St_G(u) \cap \St_G(v)$ acts transitively on $D(u)\times D(v)$.
\end{proof}

\begin{example}\label{ex:locally-not-distance}
  Here we construct, using iterated wreath products, groups which act locally
  $2$-transitively, but not distance transitively.
  
  Let $H$ be a finite group with $|H| > 2$. We define $X_i = H$ for all $i$. The associated
  tree $\calT$ is the rooted $|H|$-regular tree. Let $G_i = H \subseteq \Aut(X_i)$, where $H$
  acts on itself by multiplication from the left. The collection $C = (H | i \in \bbN)$ defines the
  iterated wreath product $G_C = \dots H \wr H \wr H$ group.
  By Lemma \ref{lem:wreath-locally2} the group $G_C$ acts locally $2$-transitively on $\calT$.
  
  We verify that $G_C$ does not act distance transitively. 
  Since $|H| \geq3$ we find pairwise distinct elements  $x,y,z \in H$.
  We obtain three vertices in $L_2$: $u_1= v_1 =(x,x)$, $u_2 = (y,x)$ and $v_2=(z,x)$
  with $d(u_1,u_2) = 4 = d(v_1,v_2)$.
  Every element $h \in H$ which satisfies $hx=x$ is already the identity element. Therefore,
  an element $g \in G_C$ with $gu_1 = u_1$ has label $g_{(\varepsilon)} = \id$
  and acts trivially on the first level.
  In particular, $gu_2 \neq v_2$ and we deduce that $G_C$ does not act distance transitively.
\end{example}

\begin{example}\label{ex:dihedral}
  We discuss an example of a group which acts spherically transitively but not locally
  $2$-transitively.  Let $\calT$ be the rooted binary tree, i.e. $X_i = \{0,1\}$ for all $i$.
  We consider the group $G$ which is generated by two elements $a,b \in \calG$.  The
  automorphism $a$ simply swaps the two elements of the first level, i.e.
  \begin{equation*}
      a (u_1x) = (1-u_1)x.
    \end{equation*}
    for all $u_1\in \{0,1\}$ and $x \in L_n$.
   The automorphism $b$ is defined recursively by the rule
   \begin{equation*}
     b (u_1x) = \begin{cases}
       (0) a(x) \quad \text{ if } u_1 = 0\\
       (1) b(x) \quad \text{ if } u_1 = 1
        \end{cases}
 \end{equation*}
 for $u_1\in\{0,1\}$ and $x \in L_n$.  The group $G$ is infinite and generated by two elements
 of order two, which implies that $G$ is isomorphic to the infinite dihedral group
 $D_\infty$. Clearly, $G$ acts transitively on the first level $L_1$. Using induction it follows that
 $G$ acts indeed spherically transitive. We will see that the action of $G$ on
 $\calT$ is not locally $2$-transitive.
   
  Consider the vertices on the second level and their descendants.
  \begin{figure}[h]
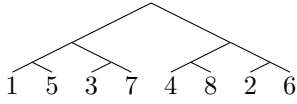

  \Tree[ [ [ [.1 ]
                        [.5 ]]
                      [ [.3 ]
                        [.7 ]]]
                    [ [ [.4 ]
                        [.8 ]]
                      [ [.2 ]
                        [.6 ]]]]
  \caption{A numbering of the vertices on the third level.}\label{fig:tree-numbering}
 \end{figure}

 If we number the vertices on the third level as in Figure \ref{fig:tree-numbering}, then the
 generators $a$ and $b$ act on the third level like reflections of the regular octagon as
 indicated in Figure \ref{fig:octagon}. In particular, the action factors through the dihedral
 group $D_8$.
\begin{figure}[h]
  \begin{tikzpicture}
    \node[draw,thick, minimum size=2.5cm,regular polygon,regular polygon sides=8] (a) at (0,0) {};
    \node[circle,label=above:{$1$}, radius=2pt] at (a.corner 2) {};
    \node[circle,label=above:{$2$}, radius=2pt] at (a.corner 1) {};
    \node[circle,label=right:{$3$}, radius=2pt] at (a.corner 8) {};
    \node[circle,label=right:{$4$}, radius=2pt] at (a.corner 7) {};
    \node[circle,label=below:{$5$}, radius=2pt] at (a.corner 6) {};
    \node[circle,label=below:{$6$}, radius=2pt] at (a.corner 5) {};
    \node[circle,label=left:{$7$}, radius=2pt] at (a.corner 4) {};
    \node[circle,label=left:{$8$}, radius=2pt] at (a.corner 3) {};
    \draw[thick, red] (1.3,1.3) -- (-1.3,-1.3);
    \node[circle, label=above right:{$a$}] at (1.2,1.2) {};
    \draw[thick, blue] (0.704,1.699) -- (-0.704,-1.699);
    \node[circle, label=below:{$b$}] at (-0.73,-1.6) {};
\end{tikzpicture}
\caption{The action on the third level.}\label{fig:octagon}
\end{figure}
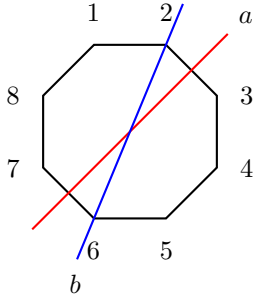
The vertices of the second level correspond to pairs of opposite vertices in the octagon.
Now, if we choose two neighbouring pairs of opposite vertices, say $\{1,5\}$ and $\{2,6\}$, then
the intersection of the two stabilisers consists merely of the identity and the point reflection in the center.
Clearly, this group with two elements cannot act transitively on the $4$-element set $\{1,5\}\times \{2,6\}$. 
\end{example}

Later the following characterisation of locally $2$-transitive actions will be useful.
\begin{lemma}\label{lem:non-sym-criterion}
  Let $G \leq \calG$ be a spherically transitive group.
  Choose a vertex $u_n \in L_n$ for every $n\in \bbN_0$.
  The group $G$ acts locally $2$-transitively on $\calT$
  if and only if for every $n\in \bbN$ and every $v\in L_n$ with $u_{n+1} \notin D(v)$
  the group $\St_G(u_{n+1}) \cap \St_G(v)$ acts transitively on $D(v)$.
\end{lemma}
\begin{proof}
  Assume first that $G$ acts locally $2$-transitively
  and let $v \in L_n$ ($n\geq 1$) be given such that $u_{n+1}\notin D(v)$.
  Let $u \in L_n$ be the predecessor of $u_{n+1}$ and observe that $u\neq v$.
  By assumption the group $\St_G(u)\cap \St_G(v)$ acts transitively on $D(u)\times D(v)$, thus
  $\St_G(u_{n+1})\cap\St_G(v)$ acts transitively on $D(v)$.

  Conversely, let $x,y \in L_n$ with $x\neq y$ be given.
  Let $(a,b), (c,d) \in D(x)\times D(y)$.
  Since the action of $G$ is spherically transitive, there is $g\in G$ such that $ga = u_{n+1}$.
  By assumption there is an element $h \in \St_G(u_{n+1}) \cap \St_G(gy)$ such that $hgb = gd$.
  Thus $\gamma_1 = g^{-1} h g \in \St_G(a) \cap \St_G(y)$ satisfies $\gamma_1 b = d$.
  By the same argument, we find $\gamma_2 \in \St_G(d) \cap \St_G(x)$ such that $\gamma_2a = c$.
  Now $\gamma_2\gamma_1 \in \St_G(x) \cap\St_G(y)$ satisfies $\gamma_2\gamma_1a = c$ and $\gamma_2\gamma_1 b = d$.
  We deduce that the action of $\St_G(x)\cap \St_G(y)$ on $D(x) \times D(y)$ is transitive. 
\end{proof}
  
\subsection{The representation on the boundary}

Let $G \leq \calG$. We fix a field $F$ whose characteristic does not divide the Steinitz order of
$\overline{G}$. 
In the case where $G$ acts spherically transitive, this assumption implies
that $\chr(F) \nmid m_n$ for every $n\in \bbN$.
For every $n\in \bbN$ the group $G$ (and $\overline{G}$) acts on the
level $L_n$ and we obtain, as above, an associated representation $\rho_{n,F}$ of $\overline{G}$ on the space
\begin{equation*}
   F[L_n] = \Bigl\{\sum_{u \in L_n} c_u u \mid c_u \in F \Bigr\}.
 \end{equation*}
Similarly, for every vertex $u\in L_n$ the stabiliser $\St_{G}(u)$ acts on the descendants $D(u)$
and we obtain an associated representation $\theta_{u,F}$ of $\St_G(u)$ on $F[D(u)]$.  As $G$ acts on
the boundary $\partial\calT$, we obtain a natural representation $\rho_{\partial,F}$ of $\overline{G}$
on the space $\Cinf(\partial\calT,F)$ of $F$-valued locally constant functions on
$\partial\calT$. More precisely, we define
\begin{equation*}
   \rho_{\partial,F}(g)(f)(\xi) := f(g^{-1}\xi)
\end{equation*}
for all $g\in \overline{G}$, $f\in C^\infty(\partial\calT,F)$ and $\xi\in \partial\calT$.  In fact, the
boundary representation can be described in various ways. First, it is the direct limit of the level
representations $\rho_{n,F}$, i.e.,
\begin{equation*}
   \rho_{\partial,F} = \varinjlim_{n\in \bbN} \rho_{n,F}.
\end{equation*}
Moreover, if $G$ is spherically transitive, then $\partial\calT \cong \overline{G}/P$
for any parabolic subgroup $P$ of $\overline{G}$.
It follows from the definition of
smooth induction (see Section~\ref{sec:rep-profinite}) that
$\rho_{\partial,F} = \smInd^{\overline{G}}_P(\triv_F)$.

 Let $v \in V(\calT)$ and let $(\theta^0_{v,F},F[D(v)]^0)$
denote the complement of the one-dimensional trivial subrepresentation inside
$(\theta_{v,F},F[D(v)])$. As a vector space
$$F[D(v)]^0 = \Bigl\{ \sum_{w \in D(v)} c_w w \mid \sum_{w\in D(v)} c_w = 0 \Bigr\}.$$
Let $\Irr_F^0(v)$
denote the set of irreducible $F$-representations of $\St_G(v)$ which occur with non-zero
multiplicity in~$\theta_{v,F}^0$.

\begin{lemma}\label{lem:irreducible}
  Assume that $G$ acts locally $2$-transitively on $\calT$. Let $\pi\in \Irr_F^0(v)$ and
  $\sigma \in \Irr_F^0(w)$ for vertices $v,w \in V(\calT)$. Then the following statements hold.
  \begin{enumerate}
  \item\label{it:Ind-irreducible} The induced representation $I_v(\pi) = \Ind_{\St_G(v)}^G(\pi)$ is
    irreducible.
   \item\label{it:dim-formula} If $v \in L_n$, then $\dim_F I_v(\pi) = |L_n|\dim_F \pi$.
   \item\label{it:distinct} $I_w(\sigma) \cong I_v(\pi)$ if and only if $v = tw$ and
     $\pi \cong \up{t}\sigma$ for some $t \in G$.
   \end{enumerate}
   Here $\up{t}\sigma$ is the representation of $t\St_G(w)t^{-1}$
   given by $\up{t}\sigma(g) = \sigma(t^{-1}gt)$. 
\end{lemma}
\begin{proof}
  Since $G$ acts spherically transitively, we have $|G:\St_G(v)| = |L_n|$ and assertion
  \eqref{it:dim-formula} follows.
  
  For $t\in G$ we write $H_t= \St_G(v) \cap \up{t}\St_G(w)$.
  We apply Mackey's intertwining number theorem (see (10.23) in \cite{CurtisReinerI}):
  \begin{equation*}
    \dim_F \Hom_{F[G]}(I_v(\pi),I_w(\sigma)) = \sum_{t \in R}
    \dim_F \Hom_{F[H_t]}(\pi|_{H_t},\up{t}\sigma|_{H_t})
  \end{equation*}
  where the sum runs over a set $R\subseteq G$ of representatives for the double cosets
  $ \St_G(v)\backslash G/\St_G(w)$.

  Assume first that $v$ and $w$ lie in the same level of $\calT$. Fix $t\in R$ and define
  $u=tw$ and $H = H_t$.  Clearly, if $u=v$, then $H = \St_G(v)$ and
  $\dim_F \Hom_{F[H]}(\pi|_{H},\up{t}\sigma|_{H})$ is non-zero exactly, when
  $\up{t}\sigma \cong \pi$. In this case the dimension is the dimension
  of the endomorphism $F$-algebra $\End_{F[H]}(\pi)$.

  Assume now that $u\neq v$.  The action of $G$ is locally
  $2$-transitive, this means that the group $H = \St_G(v) \cap \St_G(u)$ acts transitively on
  $D(v) \times D(u)$.  The representation $\pi|_H$ is a subrepresentation of
  $(\theta_{v,F}|_H, F[D(v)])$ and $\up{t}\sigma|_H$ is a subrepresentation of
  $(\theta_{u,F}|_H,F[D(u)])$.  By Lemma \ref{lem:diagonal-transitive} the trivial representation
  is the unique irreducible representation of $H$ which occurs in both $F[D(v)]$ and in
  $F[D(u)]$.  However, by assumption $\pi$ is non-trivial and since $H$ acts transitively on
  $D(v)$, we deduce that $\pi|_H$ does not admit a non-zero $H$-fixed vector. We deduce that
  $\pi|_H$ and $\up{t}\sigma|_H$ do not share an irreducible constituent; this implies
  $\dim_F \Hom_{F[H]}(\pi|_{H},\up{t}\sigma|_{H}) = 0$.
  In particular, this yields $\dim_F \End_{F[G]}(I_v(\pi))) = \dim_F \End_{\St_G(v)}(\pi)$
  and we deduce \eqref{it:Ind-irreducible} from Lemma \ref{lem:irreducibilityCriterion}.
  Now the computation above implies \eqref{it:distinct} provided $v$ and $w$ lie in the same level.
  
  Finally, assume that $v$ and $w$ lie in different levels, say $v \in L_n$ and $w\in L_m$
  with $n < m$. We observe that
  $\dim_F I_v(\pi) = |L_n| \dim_F \pi \leq |L_n|(m_{n+1}-1) < |L_{n+1}| \leq \dim_F
  I_w(\sigma)$, and thus the representations are not isomorphic.
\end{proof}


  \begin{proof}[Proof of Theorem \ref{thm:decomposition} and Theorem \ref{thm:local-global}]
    Recall that the representation $\rho_{\partial,F}$ on the locally constant functions on the
    boundary $\partial \calT$ is the direct limit of the representations $\rho_{n,F}$ of $G$ on
    $F[L_n]$.
    We will proceed by induction on $n$ to show that
    \begin{equation}\label{eq:induction-formula}
      \rho_{n,F} \cong \triv_F \oplus \bigoplus^{n-1}_{j = 0} \bigoplus_{\pi\in\Irr_F^0(v_j)}m(\pi,\theta_{v_j,F})\;  I_{v_j}(\pi).
    \end{equation}
    Clearly, $\rho_{0,F} = \triv_F$ is the trivial representation.
    Assume that formula \eqref{eq:induction-formula} holds and consider $\rho_{n+1,F}$.
    Let $u \in D(v_n)$ be a descendant of $v_n$.
    Then the following short calculation completes the induction step:
    \begin{align*}
      \rho_{n+1,F} &\cong \Ind^G_{\St_G(u)}(\triv_F) \cong \Ind^G_{\St_G(v_n)}\Ind^{\St_G(v_n)}_{\St_G(u)}(\triv_F) \\
      &\cong \Ind_{\St_G(v_n)}^G(\theta_{v_n,F})
      \cong \Ind^G_{\St_G(v_n)}(\triv_F) \oplus \Ind_{\St_G(v_n)}^G(\theta^0_{v_n,F})\\
      &\cong \rho_{n,F} \oplus \bigoplus_{\pi \in \Irr_F^0(v_n)} m(\pi,\theta_{v_n,F}) I_{v_n}(\pi).
    \end{align*}
     By Lemma \ref{lem:irreducible} the irreducible representations $I_{v_n}(\pi)$ are pairwise
    distinct, therefore the multiplicities of irreducible representations in $\rho_{\partial,F}$
    are exactly the multiplicities $m(\pi,\theta_{v_n,F})$.  We deduce that $\rho_{\partial,F}$ is
    multiplicity-free if $\theta_{v,F}$ is multiplicity-free for every $v \in V(\calT)$.
    The converse holds as well since $v_n\in L_n$ was arbitrary.
  \end{proof}
  
  Now we specialise to the case $F=\bbC$.  The following result provides a simple method to
  construct Gelfand pairs of profinite groups.

\begin{corollary}\label{cor:prime-square}
  Assume that for every $i\in \bbN$ the number $m_i$ is either a prime or the square of a
  prime.  Let $G \leq_c \calG$ be a closed subgroup which acts locally $2$-transitively on
  $\calT$.  Then $(G,P_\xi)$ is a Gelfand pair for every parabolic subgroup $P_\xi \subseteq G$.
\end{corollary}
\begin{proof}
  By the local-global principle, we need to verify that $(\St_G(v),\St_G(u))$ is a Gelfand
  pair for every vertex $v$ and descendant $u\in D(v)$.  Let $H$ denote the image of
  $\St_G(v)$ in the symmetric group $\Sym(D(v))$.  We observe that it suffices to verify that
  $H$ contains an abelian transitive subgroup $A$.  Indeed, the representation of $A$ on
  $\bbC[D(v)]$ is multiplicity-free and thus the representation of $H$ on $\bbC[D(v)]$ cannot
  contain an irreducible constituent of multiplicity exceeding $1$.

  \smallskip

 \emph{Case 1}: $|D(v)|=p$ is a prime number.\\
  Then $H$ contains an element $g$ of order $p$, i.e.\ a $p$-cycle. We
  set $A = \langle g \rangle$.

  \smallskip
  
  \emph{Case 2}: $|D(v)| = p^2$ is the square of a prime number.\\
  Let $K \leq H$ be a Sylow $p$-subgroup, say $|K| = p^r$. We note that $K$ acts transitively on $D(v)$.
  Indeed, suppose that $|K/K\cap\St_H(u)| \leq p$, then $p^{r-1}$ divides $|\St_H(u)|$ and so
  $p^{r+1}$ divides $|\St_H(u)| p^2 = |H|$. This contradicts the fact that $K$ is a Sylow $p$-subgroup.
  Let $g\in Z(K)$ be a non-trivial element in the center of $K$.
  If $\langle g \rangle$ acts transitively, we can choose $A = \langle g \rangle$.
  Otherwise, we find and element $w \in D(v)$ such that $O = \{g^k w \mid k\in \bbZ\}$ has exactly $p$ elements.
  The action of $K$ on $D(v)$ is transitive, so there is some $h\in K$ such that $hw \notin O$.
  We conclude that the group $A = \langle g,h \rangle$ is abelian and acts transitively on $D(v)$.
\end{proof}

\begin{example}\label{ex:wreath-not-Gelfand}
  We return to the setting of Example \ref{ex:locally-not-distance}
  where $G_C = \dots H \wr H \wr H$ is the iterated wreath product of
  some finite group $H$ with $|H| \geq 3$.
  Note that the group $G_C$ is a branch group; see \cite[\S 1.3]{BarGriSun2003}.
  If the group $H$ is non-abelian, then  $(G_C,P_\xi)$ is not a Gelfand pair
  for every $\xi \in \partial \calT$.
  
  For simplicity, we discuss only for the symmetric group $H = S_3$.  In this case
  $X_i=S_3$ is the symmetric group on $3$-letters and let $G_i = S_3 \subseteq \Aut(X_i)$,
  where $S_3$ simply acts on itself by multiplication from the left. Therefore $\calT$ is the
  $6$-regular rooted tree. For every vertex $u$, the action of $\St_G(u)$ on $D(u)$
  factors through the regular representation of $S_3$.

  The regular representation of $S_3$ on $\bbC[S_3]$ decomposes into two one-dimensional
  representations, namely the trivial representation $\triv = $ {\tiny\yng(3)} and the
  sign character~{\tiny\Yvcentermath1$\yng(1,1,1)$}, and a two-dimensional irreducible representation 
   $\tau =$ {\tiny\Yvcentermath1$\yng(2,1)$} which occurs twice. By Theorem \ref{thm:decomposition} the irreducible
  representations $I_v(\tau)$ occur with multiplicity two in
  $\rho_{\partial,\bbC} = \Ind_{P_\xi}^{G_C}(\triv)$ for every vertex $v$ where $P_\xi$ denotes a
  parabolic subgroup of $G_C$. We conclude that $(G_C, P_\xi)$ is not a Gelfand pair.
\end{example}

\subsection{An application to  zeta functions of representations}\label{sec:repr-zeta-functions}
The aim of this section is to use Theorem \ref{thm:decomposition} to compute the zeta function of an induced
representation. The definition given in the introduction suffices for our purposes. For more information
on zeta functions of induced representations the reader should consult~\cite{KiKl}.

 Let $R$ be a compact discrete valuation ring
 with maximal ideal $\frakp \subseteq R$ and uniformiser $\pi \in \frakp$.
 The residue field $k = R/\frakp$ is finite, say of cardinality $q$.
 We note that the ring $R$ is the ring of integers
 of some non-archimedean local field (i.e., $\bbQ_p$, $\bbF_p(\!(t)\!)$ or finite extensions thereof).
 
 The general linear group $\GL_{n}(R)$ is a profinite group.
 For every $r \in \bbN$, we define the \emph{principal congruence subgroup}
 \begin{equation*}
   \GL^r_{n}(R) = \ker\left( \GL_{n}(R) \to \GL_{n}(R/\frakp^r)\right)
 \end{equation*}
 of level $r$ as the kernel of the reduction-mod-$\frakp^r$ homomorphism.
 Let $P$ denote the parabolic subgroup of $\GL_n(R)$ which consists of all matrices
 $A = (a_{i,j}) \in \GL_{n}(R)$ with $a_{i,1} = 0$ for all $i\geq 2$.
 We put $P^r = P \cap \GL^r_{n}(R)$.

 \begin{proposition}
   In the above setting let $G = \GL^1_{N+1}(R)$.  The group $G$ acts locally $2$-transitively on
   a $q^{N}$-regular rooted tree $\calT$ such that $P^1$ is the stabiliser of a point
   $\xi \in \partial \calT$.  The pair $(G,P^1)$ is Gelfand and the zeta function of the smoothly
   induced representation $\rho = \smInd^G_{P^1}(\triv_\bbC)$ is
   \begin{equation}\label{eq:zeta-formula}
      \zeta_\rho(s) = q^{N} \frac{1- q^{-(s+1)N}}{1-q^{-sN}}.
   \end{equation}
 \end{proposition}
 \begin{remark}
   We observe that formula \eqref{eq:zeta-formula} only depends on the residue field
   cardinality $q$ and is, in particular, independent of the characteristic of $R$.  This
   phenomenon was observed in several different cases and calls for a deeper investigation;
   see also \cite[Problem~1.1]{KiKl}.
   Moreover, we note that formula \eqref{eq:zeta-formula} was known to hold 
    if $R$ is of characteristic $0$ \cite[Prop.~7.7]{KiKl}.
 \end{remark}
 \begin{proof}
   Let $S$ be a commutative ring. An element $(x_1,\dots,x_{N+1})\in S^{N+1}$ is called \emph{primitive},
   if the ideal generated by $x_1,\dots,x_{N+1}$ is $S$.
   We denote by $\bbP^{N}(S)$ the $S$-rational points of the $N$-dimensional projective space,
   that is
   \begin{equation*}
       \bbP^{N}(S) = \{x \in S^{N+1} \mid x \text{ is primitive }\} / S^\times.
     \end{equation*}
     The equivalence class of
     $(x_1,\dots,x_{N+1})$ in $\bbP^{N}(S)$ will be denoted using projective coordinates as
     $(x_1\colon x_2\colon\dots\colon x_{N+1})$.

     Consider the tree $\calT$ whose vertices of level $n\in \bbN_0$ are the elements of
     \begin{equation*}
       L_n = \{ v \in \bbP^{N}(R/\frakp^{n+1}) \mid v \equiv (1\colon 0 \colon\dots\colon 0)\bmod\frakp\}.
     \end{equation*}
     There is an edge between $u \in L_{n-1}$ and $v \in L_n$ if and only if
     $v \equiv u \bmod \frakp^n$.  We define $u_n \in L_n$ to be the point
     $u_n= (1\colon 0 \colon\dots\colon 0) \in \bbP^{N}(R/\frakp^{n+1})$.  The root of
     $\calT$ is the vertex $u_0$.  The boundary $\partial \calT$ can be (and will be)
     identified with
     $\partial \calT = \{ x \in \bbP^{N}(R) \mid x \equiv u_0 \bmod \frakp \}$.  We fix
     $\xi = (1\colon 0 \colon \dots \colon 0)\in \partial\calT$.  For every $n \in \bbN$ the
     group $\GL_{N+1}(R)$ acts on $\bbP^{N}(R/\frakp^{n})$ and the restrictions of these actions
     to $\GL^1_{N+1}(R)$ provide an action on $\calT$. Observe that $P^1$ is the stabiliser of~$\xi$.

     \smallskip

     \emph{Step 1:} The action of $G$ on $\calT$ is spherically transitive.\\
     Let $v \in L_n$, say $v = (1\colon v_2+\frakp^{n+1}\colon \dots \colon v_{N+1}+ \frakp^{n+1})$
     with certain $v_i \in \frakp$.
     Then $g u_n = v$ where
     \begin{equation*}
       g = \begin{pmatrix}
         1 & 0 & \cdots & 0\\
         v_2 & 1 & & \\
         \vdots & &\ddots & \\
         v_{N+1} & & & 1 \\
         \end{pmatrix} \in \GL_{N+1}^1(R).
       \end{equation*}

       \smallskip

       \emph{Step 2:} The action of $G$ on $\calT$ is locally $2$-transitive.\\
       Let $v\in L_{n-1}$ with $v \neq u_{n-1}$ for some $n\geq 2$.
       Since $G$ acts spherically transitive, it is
       sufficient to show that $\St_G(u_{n})\cap \St_G(v)$ acts transitively on the descendants
       $D(v)$; see Lemma \ref{lem:non-sym-criterion}.
       Let $x,y \in D(v) \subseteq L_n$.
       There are $1\leq \ell < n$ and primitive vectors
       $\tilde{x} = (x_2,\dots,x_{N+1})$ and $\tilde{y} = (y_2,\dots,y_{N+1}) \in R^{N}$
       such that $x = (1\colon \pi^\ell x_2 + \frakp^{n+1}\colon \dots\colon \pi^\ell x_{N+1} + \frakp^{n+1})$
       and $y = (1\colon \pi^\ell y_2 + \frakp^{n+1}\colon \dots\colon \pi^\ell y_{N+1} + \frakp^{n+1})$.
       Indeed, the exponent $\ell$ is the largest number such that $v$ (as well as $x$ and $y$)
       hangs below $u_{\ell-1}$.

       Since the primitive vectors $\tilde{x}$ and $\tilde{y}$ satisfy
       $\tilde{x} \equiv \tilde{y} \bmod \frakp^{n-\ell}$, there is a matrix $A \in \GL_N^{n-\ell}(R)$
       such that $A\tilde{x} = \tilde{y}$.
       Moreover, $n-\ell \geq 1$ and thus the matrix
       \begin{equation*}
         g = \begin{pmatrix}
           1 & 0 & \cdots & 0\\
           0& & & &\\
           \vdots & & A &\\
           0 & & & &
           \end{pmatrix}
         \end{equation*}
         is an element of $\St_G(u_{n}) \cap \St_G(v)$ such that $gx = y$.

         \smallskip

         \emph{Step 3:} The representation of $\St_G(u_n)$ on $\bbC[D(u_n)]$ decomposes
         into pairwise distinct $1$-dimensional irreducible representations.\\
         Let $n \in \bbN_0$. We consider the bijective map
         $D(u_n) \to k^{N}$ which maps
         an element $(1\colon \pi^{n+1}x_2\colon \dots \colon \pi^{n+1} x_{N+1})$
         to $(x_2+\frakp, \dots, x_{N+1}+\frakp)\in k^{N}$.
         A short calculation shows that via this identification the stabiliser $\St_G(u_n)$
         acts like the additive group $k^{N}$ on itself by translations.
         More precisely, $\St_G(u_{n+1})$ is normal in $\St_G(u_n)$ and the quotient is isomorphic
         to $k^N$.
         In particular,
         $\bbC[D(u_n)]$ decomposes like the regular representation of the abelian group $k^{N}$.

         \smallskip

         Finally, we deduce from Theorem \ref{thm:decomposition} that $(G,P^1)$ is a Gelfand pair.
         Moreover, the representation on the boundary $\rho_{\partial,\bbC}$
         is isomorphic to the induced representation $\smInd_{P^1}^G(\triv_\bbC)$.
         It decomposes into a direct sum of the trivial representation and, for every $n\in \bbN_0$,
         exactly $q^{N}-1$ irreducible representations of dimension $q^{nN}$.
         Using the geometric series we obtain
         \begin{equation*}
           \zeta_{\rho}(s) = 1 + \sum_{n = 0}^\infty (q^{N}-1) q^{-snN}
           =  q^{N} \frac{1- q^{-(s+1)N}}{1-q^{-sN}}. \qedhere
         \end{equation*}
 \end{proof}

 \section{Local $2$-transitivity of GGS-groups} \label{sec:GGS} In this section we look at a
 famous family of examples: the GGS-groups.  For more information about these groups the
 reader may consult \cite[Sec.~2.3]{BarGriSun2003}. One aim of this section is to prove that
 GGS-groups which act on a $p$-regular rooted tree are locally $2$-transitive provided that
 $p$ is odd (see Theorem \ref{thm:GGSsimpleversion}).  In fact, we will work in greater
 generality and study generalised GGS-groups which act on $p^k$-regular trees. Theorem
 \ref{thm:GGS-2-trans} provides a criterion for local $2$-transitivity of such groups.

Let $p$ be a prime number and let $k\geq 1$ be an integer. We define $X= \bbZ/p^k\bbZ$.  In
this section we define $X_i = X$ for all $i\in \bbN$. The associated tree $\calT$ is 
the \emph{$p^k$-regular rooted tree}. For convenience, we put $m = p^k$.  We note that the
tree hanging from a vertex $u$ is isomorphic to $\calT$ itself.  For every
$g \in \St_\calG(u)$ the restriction of $g$ to the subtree hanging from $u$ is an element
$g_u \in \calG$. In particular, we obtain an isomorphism
\begin{equation*}
  \psi\colon \calG(1) \to \underbrace{\calG \times \cdots \times \calG}_{m \text{ times }}
\end{equation*}
defined as $\psi(g) = (g_{1+m\bbZ},\dots,g_{m+m\bbZ})$.

Let $\alpha \in \Sym(X)$ be the cyclic permutation $\alpha(u) = u+1$.
Let $a \in \calG$ be the automorphism which permutes the vertices on the first level, and
consequently also the trees hanging from them, cyclically according to $\alpha$, i.e.\
the label at the root is $a_{(\epsilon)} = \alpha$.

Let $\fat{e} = (e_1,\dots,e_{m-1}, e_m) \in (\bbZ/m\bbZ)^{m}$. We say that $\fat{e}$ is a
\emph{defining vector} if $e_m =0$ and $e_i \not\equiv 0 \bmod p$ for some index $i$.  In most
references on GGS-groups the defining vector is of length $m-1$ and the last entry $e_m=0$ is
ignored. For our purposes, however, it is very convenient to consider the additional entry
$e_m=0$ to be part of the defining vector. In addition, it is convenient to define
$e_{i+m\ell} = e_i$ for all $\ell\in \bbZ$, such that we are able to consider the indices of
$e_1,\dots,e_m$ modulo $m$. We use a defining vector $\fat{e}$ to define recursively an
automorphism $b_{\fat{e}} \in \calG(1) \subseteq \calG$ by imposing
\begin{equation*}
    \psi(b_{\fat{e}}) = (a^{e_1}, a^{e_2}, \dots, a^{e_{m-1}}, b_{\fat{e}}).
\end{equation*}

\begin{definition}
  Let $\fat{e}$ be a defining vector.
  The subgroup $G_{\fat{e}}$ of $\calG$ generated by $a$ and $b_{\fat{e}}$ is called the \emph{GGS-group} defined
  by $\fat{e}$.
\end{definition}
Whenever the defining vector is clear from the context, we write $G$ instead of $G_{\fat{e}}$
and $b$ instead of $b_{\fat{e}}$.

\subsection{Aperiodic defining vectors}
We say that a defining vector $\fat{e}$ is \emph{periodic modulo $p$}, if
there is a $t \in \{1,\dots, m-1\}$ such that $e_i \equiv e_{i+t} \bmod p$ for all $i$.
Otherwise, $\fat{e}$ is said to be \emph{aperiodic modulo $p$}.
For every $t\in \{1,\dots,m-1\}$ we define the $(2\times m)$-matrix
\begin{equation*}
  A_t = \begin{pmatrix}
    e_1 & \dots & e_{m}\\
    e_{1+t} & \dots &e_{m+t}\\
  \end{pmatrix}
\end{equation*}
  with entries in the local ring $\bbZ/m\bbZ$.
  The reduction of $A_t$ modulo $p$ will be denoted by $\overline{A}_t$.
 
  \begin{lemma}\label{lem:CharacterizeAperiodic}
    Let $\fat{e} \in (\bbZ/m\bbZ)^m$ be a defining vector.
    The following statements are equivalent.
    \begin{enumerate}
    \item\label{it:aperiodic} $\fat{e}$ is aperiodic modulo $p$.
    \item\label{it:Abar-surjective} $\overline{A}_t \in M_{2\times m}(\bbF_p)$ has rank $2$
      for all $t\in\{1,\dots,m-1\}$.
    \item\label{it:A-surjective} Left multiplication with $A_t$ defines a surjective $\bbZ/m\bbZ$-linear map
      from
      $(\bbZ/m\bbZ)^m$ onto $(\bbZ/m\bbZ)^2$
      for all $t\in\{1,\dots,m-1\}$.
    \end{enumerate}
  \end{lemma}
  \begin{proof}
    \eqref{it:aperiodic} $\implies$ \eqref{it:Abar-surjective}: Assume that $\overline{A}_t$
    does not have rank $2$.  Then the two rows are linearly dependent and there is
    $\lambda \in \bbF_p$ such that $e_i \equiv \lambda e_{i+t} \bmod p$
    for all $i$.  We infer that
    $e_i\equiv \lambda^{m} e_{i+mt} \equiv \lambda^{p^k} e_i \bmod p$  for all $i\in\{1,\dots,m\}$.  By
    assumption $e_j \not\equiv 0 \bmod p$ for some $j$ and we conclude $\lambda = 1$,
    i.e., $\fat{e}$ is periodic modulo $p$.

    \eqref{it:Abar-surjective} $\implies$ \eqref{it:A-surjective}: Since $m = p^k$, the ring
    $\bbZ/m\bbZ$ is a local ring.  In particular, it follows from Nakayama's lemma that the linear map defined by $A_t$ is
    surjective onto $(\bbZ/m\bbZ)^2$ if and only if the linear map defined by $\overline{A}_t$
    is surjective onto $\bbF_p^2$.

    \eqref{it:A-surjective} $\implies$ \eqref{it:aperiodic}: Let $t\in\{1,\dots,m-1\}$. By
    assumption there is a vector $c = (c_1,\dots,c_m) \in (\bbZ/m\bbZ)^m$ such that
    $\sum_{i=1}^m e_ic_i = 1$ and $\sum_{i=1}^m e_{i+t}c_i = 0$.  In particular, since
    $0 \not\equiv 1 \bmod p$, we deduce that $e_j \not\equiv e_{j+t} \bmod p$ for some
    $j\in\{1,\dots,m\}$.
  \end{proof}

  \begin{remark}\label{rem:p-always-aperiodic}
    Assume that $k=1$, that is $m = p$, where $p$ is a prime number.  In this case every defining vector
    $\fat{e}$ is aperiodic (modulo $p$).  Indeed, suppose that $e_i = e_{i+t}$ for all $i\in \{1,\dots,p\}$
    and some $t\in \{1,\dots,p-1\}$.  The element $t+p\bbZ$ generates the cyclic group
    $\bbZ/p\bbZ$ and we conclude that $e_i = e_j$ for all $i$ and $j$. This yields a contradiction
    since $e_p = 0$, whereas by assumption $e_i \neq 0$ for some $i \in \{1,\dots,p-1\}$
  \end{remark}

  \subsection{A criterion for local $2$-transitivity of GGS-groups}
  Let $p$ be a prime number, let $m = p^k$ and let $\fat{e} \in (\bbZ/m\bbZ)^m$ be a defining vector.
  We consider the GGS-group $G = G_{\fat{e}}$  and its subgroup
  \begin{equation*}
      H =\bigl\langle [b^k,a^\ell] \mid k,\ell \in \{1,\dots,m-1\}\bigr\rangle.
  \end{equation*}
  Here $[b^k,a^\ell] = b^k a^{\ell}b^{-k}a^{-\ell}$ denotes the commutator of $b^k$ and $a^\ell$.
  We write $b_i = a^i b a^{-i}$, so that
  $\psi(b_i) = (a^{e_{1-i}}, \dots ,a^{e_{-1}},\underbrace{b}_{i},a^{e_1},\dots,a^{e_{m-i}})$.
  In particular, we have the identity
  \begin{align}\label{eq:commutator-formula}
    \psi([b^k,a^\ell]) &=\psi( b^k b_\ell^{-k})\nonumber\\
                       &= (a^{k(e_1-e_{1-\ell})}, \dots, a^{k(e_{\ell-1}-e_{-1})},
                         a^{ke_\ell}b^{-k},a^{k(e_{\ell+1}-e_1)},\dots,b^ka^{-ke_{m-\ell}}).
  \end{align}

  \begin{lemma}\label{lem:H-transitive}
    Let $u\in L_n$ with $n\geq 1$. The group $H \cap \St_G(u)$  acts transitively on the descendants
    $D(u)$. 
  \end{lemma}
  \begin{proof}
    We first verify that the assertion holds for $u \in L_1$.
    Say $u = t + m\bbZ$. We consider the group of labels $H_{(u)} = \{g_{(u)}\mid g \in H\}$.
    Formula \eqref{eq:commutator-formula}
    implies that
    \begin{equation*}
      H_{(u)} = \langle \alpha^{e_t-e_j} \mid j \in \{1,\dots,m\} \rangle
    \end{equation*}
    If $e_t$ generates $\bbZ/m\bbZ$, then $H_{(u)}$ contains the element $\alpha^{e_t - e_m}$
    of order $m$.  On the other hand, if $e_t$ is no generator of $\bbZ/m\bbZ$, then some $e_j$ is and
    thus also $e_t - e_j$.  We deduce that $H_{(u)}$ is the cyclic group
    generated by $\alpha$ which acts transitively on $D(u)$.  As a next step, we will show
    that $H_u = \{ g_u \mid g \in H\}$ contains $H$ for $u\in L_1$ as above.  Consider the
    following two cases.

     \smallskip
    
     \emph{Case 1}: $e_t$ does not generate the group $\bbZ/m\bbZ$.\\
     In this case $e_{t-j}$ generates $\bbZ/m\bbZ$ for some $j\in\{1,\dots,m-1\}$.  Clearly,
     then also $e_t - e_{t-j}$ is a generator.  By \eqref{eq:commutator-formula} the group
     $H_u$ is generated by $a^{e_t-e_{t-j}}$ and the elements $a^{ke_t}b^{-k}$ for all $k$.
     We conclude that $H_u = \langle a,b\rangle = G$ and thus $H\subseteq H_u$.

       \smallskip

       \emph{Case 2:} $e_t$ is a generator of $\bbZ/m\bbZ$.  In this case, we use
       \eqref{eq:commutator-formula} to deduce that the group $H_u$ contains all elements of
       the form $a^{ke_t}b^{-k}$ for $k\in \bbZ$.  A short calculation yields
    \begin{equation*}
      H_u \ni  a^{-\ell e_t}b^\ell\: b^{k-\ell}a^{(\ell-k)e_t} \:  a^{k e_t}b^{-k}
      = b^k_{-\ell e_t}b^{-k} = [a^{-\ell e_t}, b^k]
    \end{equation*}
    for all $k,\ell \in \{1,\dots,m-1\}$. Since $e_t$ is a generator of $\bbZ/m\bbZ$, we deduce
    $H \subseteq H_u$.

    Finally, it follows by induction on $n\geq 1$, that $\St_G(u)\cap H$ acts transitively on
    $D(u)$ for all $u \in L_n$.
  \end{proof}

  For $p=2$ we need to exclude a problematic case.
  We say that the defining vector $\fat{e}$ is \emph{centered},
  if $e_{m/2}$ is the only entry of $\fat{e}$ which generates $\bbZ/m\bbZ$.

  \begin{lemma}\label{lem:HtimesH}
    Let $p$ be a prime. If $p=2$, we assume that $\fat{e}$ is \emph{not} centered.
    
    Let $u, v \in L_1$ with $u\neq v$.  The image of the homomorphism
    $\Psi_{u,v} \colon [G,G] \to G\times G$ defined as $\Psi_{u,v}(g) = (g_u, g_v)$
    contains the subgroup $H\times H$.
  \end{lemma}
  \begin{proof}
    Say $u = i_0 + m\bbZ$ and $v = j_0 + m\bbZ$ with $1\leq i_0, j_0 \leq m$.
    Let $i,j \in \{1,\dots,m\} \setminus \{i_0, j_0\}$.
  Is is easy to verify that for all $k,\ell \in \bbZ$ the following identities hold:
  \begin{align}
    \Psi_{u,v}\left([b^k_i,b^\ell_{j_0}]\right)  &= \left(1, [a^{ke_{j_0-i}},b^\ell]\right)\label{eq:comm-1}, \\
    \Psi_{u,v}\left([b^k_{i_0},b^\ell_{j}]\right) &= \left([b^k,a^{\ell e_{i_0-j}}], 1\right)\label{eq:comm-2}, \\
    \Psi_{u,v}\left([b^k_{i_0},b^\ell_{j_0}]\right)  &= \left([b^k,a^{\ell e_{i_0-j_0}}], [a^{ke_{j_0-i_0}},b^\ell]\right)\label{eq:comm-3}.
  \end{align}
  Suppose that $e_{i_0-j}$ is a generator of $\bbZ/m\bbZ$ for some $j\neq j_0$.  Then equation
  \eqref{eq:comm-2} shows that the image of $\Psi_{u,v}$ contains $H \times 1$.  However, by
  \eqref{eq:comm-1} and \eqref{eq:comm-3} the projection of
  $\Psi_{u,v}([G,G])\cap (H\times H)$ onto the second factor contains $H$, therefore the claim
  follows.  By the same argument as above we can conclude if $e_{j_0-i}$ is a generator of
  $\bbZ/m\bbZ$ for some $i \neq i_0$.

  Finally, we observe that one of the two cases above applies.  Indeed, if $p$ is odd, then at
  least one entry of $\fat{e}$ is a generator and we have
  $i_0 -j_0 \not\equiv j_0 - i_0 \bmod m$.  On the other hand, suppose that $p=2$ and
  $i_0 -j_0 \equiv j_0 - i_0 \bmod m$. In this case $i_0-j_0 = \pm m/2$. Since $\fat{e}$ is
  not centered, there is some $j \neq j_0$ such that $e_{i_0-j}$ generates
  $\bbZ/m\bbZ$.
\end{proof}

\begin{theorem}\label{thm:GGS-2-trans}
  Let $p$ be a  prime number and recall that $m=p^k$.
  Let $\fat{e} \in (\bbZ/m\bbZ)^m$ be a defining vector. If $p=2$, we assume that $\fat{e}$ is not centered.
  
  The GGS-group $G_{\fat{e}}$ acts locally $2$-transitively on $\calT$ if and only if
  the vector $\fat{e}$ is aperiodic modulo $p$.
  In this case $(\overline{G}_{\fat{e}}, P_\xi)$ is a Gelfand pair for every parabolic
  subgroup $P_\xi\leq \overline{G}_{\fat{e}}$ 
\end{theorem}

\begin{proof}
  As a first step, we show that $\St_G(u)\cap \St_G(v)$ acts transitively on $D(u) \times D(v)$
  for all $u \neq v$ in $L_1$ exactly if $\fat{e}$ is aperiodic modulo $p$.

  Assume that $\fat{e}$ is aperiodic modulo $p$.
  Say $u = i + m\bbZ$ and $v= i+t+m\bbZ$ with $i\in \{1,\dots,m\}$ and $t\in\{1,\dots,m-1\}$.
  By Lemma \ref{lem:CharacterizeAperiodic} \eqref{it:A-surjective} left multiplication with
  the two row circulant matrix
  \begin{equation*}
      \begin{pmatrix}
          e_i & \dots & e_{i+m-1}\\
          e_{i+t} & \dots &e_{i+m-1+t}\\
       \end{pmatrix}
     \end{equation*}
     defines a surjective map onto $(\bbZ/m\bbZ)^2$.
     In particular, we find $c_1,\dots,c_{m} \in \bbZ/m\bbZ$
     such that
     \begin{equation*}
       \sum_{j=0}^{m-1} c_{m-j} e_{i+j} = 1 \quad \text{ and } \quad
       \sum_{j=0}^{m-1} c_{m-j} e_{i+t+j} = 0.
    \end{equation*}
    We deduce that the element $g = b_1^{c_1}b_2^{c_2}\cdots b_m^{c_m} \in G \cap \calG(1)$
    satisfies $g_{(v)}= \id$, whereas $g_{(u)}$ is the cyclic permutation $\alpha$.  Since the
    same argument applies with $u$ and $v$ interchanged, it follows that
    $G \cap \calG(1) = \St_G(u) \cap \St_G(v)$ acts transitively on $D(u) \times D(v)$.

    Conversely, let $t\in \{1,\dots,m-1\}$ be arbitrary. Choose $u = 1 +m\bbZ$ and
    $v = 1+t+m\bbZ$.  Note that an element $g \in \St_G(v)$ which fixes one descendant of $v$
    already acts trivially on $D(v)$. Since the action is locally $2$-transitive, we deduce
    that there is an element $g \in \St_G(u) \cap \St_G(v)= G \cap \calG(1)$ such that
    $g_{(v)} = \id$ but $g_{(u)}$ has order $m=p^k$. Since
    $G \cap \calG(1) = \langle b_1,\dots, b_m\rangle$ and $(G\cap\calG(1))/(G\cap\calG(2))$ is
    abelian, we see that there is some $j \in \{1,\dots,m\}$ such that
    $(b_j^{p^{k-1}})_{(u)} \neq (b_j^{p^{k-1}})_{(v)}$; in other words,
    $e_{m-j+1} \not\equiv e_{m-j+t+1} \bmod p$. We deduce that $\fat{e}$ is aperiodic
    modulo~$p$.

    Assume now that $\fat{e}$ is aperiodic modulo $p$.  The final step is to verify that for
    distinct vertices $u,v \in L_n$ the action of $\St_G(u)\cap\St_G(v)$ on $D(u) \times D(v)$
    is transitive.  We proceed by induction on $n$. The case $n=1$ was already discussed
    above.  Assume now that $n \geq 2$.  We distinguish two cases.

    \emph{Case 1}: $u$ and $v$ have the same predecessor $w \in L_1$ on the first level.
    Since $G$ is fractal (c.f.~\cite{Jone2016}), we have
    $G = G_w = \{g_w \mid g \in \St_G(w)\}$ and the claim follow from the induction
    hypothesis.

    \smallskip

    \emph{Case 2}: $u$ and $v$ have distinct predecessors $u_1\neq v_1$ on the first level $L_1$.
    This means, $u = u_1\tilde{u}$ and $v = v_1\tilde{v}$. It
    follows from Lemma \ref{lem:HtimesH} that $H \times H$ is contained in the image of the
    restriction homomorphism $\Psi_{u_1,v_1}\colon [G,G] \to G\times G$.  Now, by Lemma
    \ref{lem:H-transitive} the group $H \cap \St_G(\tilde{u})$ acts transitively on $D(\tilde{u})$ and similarly
    $H \cap \St_G(\tilde{v})$ acts transitively on $D(\tilde{v})$. This proves the assertion.

    \smallskip

    Finally, assume that $G$ acts locally $2$-transitively and let $P_\xi \leq G$ be a
    parabolic subgroup. Since $G$ is fractal, the action of the stabiliser $\St_G(u)$ on
    $D(u)$ factors through the regular representation of a cyclic group of order $m$ and thus
    the representation on $\bbC[D(u)]$ is multiplicity-free. By the local-global principle
    $(G,P_\xi)$ is a Gelfand pair.
\end{proof}

\begin{remark}\label{rem:Dinfty}
  The assumption that the
  defining vector is not centered for $p=2$ is neccessary.
  Indeed, consider the case $p=2$ and $k=1$. There
  is precisely one GGS-group $G$ acting on the rooted binary tree $\calT$; it is defined by
  the vector $(1,0) \in (\bbZ/2\bbZ)^2$. 
  Note that the
  defining vector is centered and aperiodic.
  This is exactly the group considered in
  Example \ref{ex:dihedral}, where we have seen that $G$ does \emph{not} act
  locally $2$-transitively on $\calT$. It seems possible that groups with centered defining
  vector in general do not act locally $2$-transitively.

  The closure $\overline{G}$ of $G \cong D_\infty$ in $\calG$ is isomorphic to the pro-$2$
  completion of the infinite dihedral group $D_\infty$.  The group $P= \langle b \rangle$ is a
  parabolic subgroup of $G$ and one can verify by direct calculation that $(\overline{G},P)$
  is nevertheless a Gelfand pair. It seems possible, that a weaker assumption
  still suffices the prove a local-global principle for Gelfand pairs.
\end{remark}


\begin{thebibliography}{99}
%


 \bibitem{AvKlOnVo13} N.~Avni, B.~Klopsch, U.~Onn, and C.~Voll,
   \emph{Representation zeta functions of compact {$p$}-adic analytic
     groups and arithmetic groups}, Duke Math.\ J.\ \textbf{162}
   (2013), 111--197.
%
 \bibitem{AvKlOnVo16a} \bysame, \emph{Similarity classes of integral
     $\mathfrak{p}$-adic matrices and representation zeta functions of
     groups of type $\mathsf{A}_2$}, Proc.\ London Math.\ Soc.\
   \textbf{112} (2016), 267--350.
%
%
 \bibitem{BarGri2000}
   L.\ Bartholdi, R.\ Grigorchuk,
\emph{On the spectrum of Hecke type operators related to some fractal groups},
Proc.\ Steklov Inst.\ Math.\ \textbf{231}, (2000), 1--41.

\bibitem{BarGri2001} L.~Bartholdi, and R.~Grigorchuk, \emph{On
    parabolic subgroups and Hecke algebras of some fractal groups},
  Serdica Math.\ J.\ \textbf{28} (2002), 47--90.

\bibitem{BarGriSun2003}
  L.\ Bartholdi, R.\ Grigorchuk, Z.\ \v{S}uni\'k,
\emph{Branch groups}, Handbook of Algebra, Vol.~3,
Elsevier/North-Holland, Amsterdam, 2003.

\bibitem{Bartholdi-delaHarpe2010}
  L.\ Bartholdi, P.\ de la Harpe,
  \emph{Representation zeta functions of wreath products with finite groups},
   Groups Geom.\ Dyn.\ \textbf{4} (2010), no.~2, 209--249.

\bibitem{Bartholdi2017}
  L.\ Bartholdi,
  \emph{Representation zeta functions of self-similar branched groups},
  Groups Geom.\ Dyn.\ \textbf{11} (2017), no.~1, 29--56.
  

  
  
  
\bibitem{BeHa03} M.~B.~Bekka, and P.~de~la~Harpe,
  \emph{Irreducibility of Unitary Group Representations and
    Reproducing Kernels Hilbert Spaces}, Expo.\ Math.\ \textbf{21}
  (2003), 115--149.

\bibitem{Bump} D.\ Bump
  \emph{Lie Groups}, 2nd Edition, Springer-Verlag, New York, 2013.

\bibitem{CSST2007}    T.\ Ceccherini-Silberstein, F.\ Scarabotti, F.\ Tolli,
  \emph{Finite Gel¡¯fand pairs and their applications to probability and statistics},
  J.\ Math.\ Sci.\ \textbf{141} (2007), 1182--1229.

\bibitem{CurtisReinerI} C.\ W.\ Curtis, and I.\ Reiner,
  Methods of Representation Theory, Volume I, John Wiley and Sons, New York, 1981.

\bibitem{DD2007}
 D.\ D'Angeli, A.\ Donno,
\emph{Self-similar groups and finite Gelfand pairs},
Algebra Discrete Math.\ (2007), 54--69.

\bibitem{DD2012}
  \bysame,
\emph{Appendix: Gelfand pairs associated with the action of G},
European J.\ Combin.\ \textbf{33} (2012), no.\ 7, 1422--1426.

\bibitem{Diaconis}
  P.\ Diaconis, \emph{Group Representations in Probability and Statistics},
   Institute of Mathematical Statistics, Hayward, CA, 1988.

\bibitem{DieudonneAna6}
  J.\ Dieudonn\'e,
  \emph{\'El\'ements d'analyse 6}, Chapitre XXII, Gauthier-Villars, Paris, 1975.

\bibitem{DudkoGrigorchuk2017}
A.\ Dudko, R.\ Grigorchuk,
\emph{On irreducibility and disjointness of Koopman and quasi-regular representations of weakly branch groups},
Modern theory of dynamical systems, 51--66,
Amer.\ Math.\ Soc., Providence, RI, 2017.

  \bibitem{FAZR2014}
G.\ A.\ Fern\'andez-Alcober, A.\ Zugadi-Reizabal,
\emph{GGS-groups: order of congruence quotients and Hausdorff dimension},
Trans.\ Amer.\ Math.\ Soc.\ \textbf{366} (2014), 1993--2017. 

\bibitem{FAGUA2017}
  G.\ A.\ Fern\'andez-Alcober, A.\ Garrido, J.\ Uria-Albizuri,
\emph{On the congruence subgroup property for GGS-groups}
Proc.\ Amer.\ Math.\ Soc.\ \textbf{145} (2017), 3311--3322.
  
\bibitem{Gri1980}
  R.\ Grigorchuk,
  \emph{On Burnside's problem on periodic groups} (Russian),
Funktsional.\ Anal.\ i Prilozhen.\ \textbf{14} (1980), 53--54. 

\bibitem{Gri2005}
 \bysame, \emph{Solved and unsolved problems around one group}, in:
  Infinite groups: geometric, combinatorial and dynamical aspects, Progr.\ Math.\ 248,
  Birkh\"auser, Basel, 2005.

\bibitem{GLNS2016}
  R.\ Grigorchuk, Y.\ Leonov, V.\ Nekrashevych, V.\ Sushchansky,
\emph{Self-similar groups, automatic sequences, and unitriangular representations},
Bull.\ Math.\ Sci.\ \textbf{6} (2016), 231--285. 
  

 \bibitem{Gr91} B.~H.~Gross, \emph{Some applications of Gel'fand pairs
     to number theory}, Bull.\ Amer.\ Math.\ Soc.\ \textbf{24} (1991),
   277--301.

 \bibitem{GuSi1983}
   N.\ Gupta, S.\ Sidki,
   \emph{On the Burnside problem for periodic groups},
   Math.\ Z.\ \textbf{182} (1983), 385--388. 
  

 \bibitem{Ja06} A.~Jaikin-Zapirain, \emph{Zeta function of
     representations of compact $p$-adic analytic groups},
   J.\ Amer.\ Math.\ Soc.\ \textbf{19} (2006), 91--118.

\bibitem{KiKl} S.~Kionke and B.~Klopsch, \emph{Zeta functions
    associated to admissible representations of compact $p$-adic Lie
    groups}, preprint \texttt{arXiv:1707.08485} (26 Jul 2017).

\bibitem{Kl2013}
B.\ Klopsch,
\emph{Representation growth and representation zeta functions of groups}
Note Mat.\ \textbf{33} (2013), 107--120.

\bibitem{Kowalski}
  E. Kowalski,
  \emph{An introduction to the representation theory of groups},
  American Mathematical Society, Providence, RI, 2014.

\bibitem{LangAlgebra}
  S.\ Lang, \emph{Algebra},
  3rd Ed., Springer-Verlag, New York, 2002.

\bibitem{Letac1982}
  G.\ Letac, \emph{Les fonctions sph\'eriques d'un couple de Gelfand sym\'etrique et les cha\^{i}nes de Markov},
  Adv. Appl. Prob. \textbf{14} (1982), 272--294.
  
 \bibitem{LaLu08} M.~Larsen and A.~Lubotzky, \emph{Representation
    growth of linear groups}, J. Eur. Math. Soc. \textbf{10} (2008),
   351--390.

 \bibitem{LuMa04} A.~Lubotzky and B.~Martin, \emph{Polynomial
     representation growth and the congruence subgroup growth}, Israel
   J.\ Math.\ \textbf{144} (2004), 293--316.

 \bibitem{Pervova2007}
   E.\ Pervova,
\emph{Profinite completions of some groups acting on trees}
J.\ Algebra \textbf{310} (2007), 858--879. 

\bibitem{Jone2016}
  J.\ Uria-Albizuri,
  \emph{On the concept of fractality for groups of automorphisms of a regular rooted tree},
  preprint \texttt{arXiv:1604.05950} (20 Apr 2016).

\bibitem{WilsonProfinite} J.\ S.\ Wilson,
  \emph{Profinite Groups}, Clarendon Press, Oxford, 1998.
  
  
\end{thebibliography}
\end{document}